\documentclass[a4paper,12pt]{article}
\usepackage[utf8]{inputenc}
\usepackage{amsmath,amssymb,amsthm,mathrsfs}
\usepackage{graphicx}
\usepackage{cite}
\usepackage{hyperref}
\hypersetup{
    colorlinks=true,
    linkcolor=blue
    }

\usepackage{pdflscape}
\usepackage{adjustbox}
\usepackage{setspace}
\usepackage{cellspace}

\usepackage{enumitem}
\usepackage[utf8]{inputenc}

 \usepackage{graphicx}
\usepackage{array}
\usepackage{geometry}
\usepackage{mathrsfs}

\theoremstyle{definition} \newtheorem{remark}{Remark}[section]
\theoremstyle{plain}  \newtheorem{theorem}{Theorem}[section]
\theoremstyle{plain}  \newtheorem{lemma}{Lemma}[section]
\theoremstyle{plain}  \newtheorem{proposition}{Proposition}[section]
\theoremstyle{definition} \newtheorem{definition}{Definition}[section]
\theoremstyle{remark} \newtheorem{example}{Example}[section]

\numberwithin{equation}{section}

\def \MRA{multiresolution analysis}

\title{\textbf{Vanishing Moments of Wavelets on $ p $-adic Fields}}
\author{\textbf{Athira N$ ^{\ast} $ and Lineesh M C$^{\dagger} $}}
\date{}

\begin{document}

\maketitle
\begin{center}
\small{Department of Mathematics \\
National Institute of Technology Calicut \\
NIT Campus P O - 673 601, India \\
$ ^{\ast} $nathira.1996@gmail.com \\
$ ^{\dagger} $lineesh@nitc.ac.in}
\end{center}

\begin{abstract} A study on the vanishing moments of wavelets on $ p $-adic fields is carried out in this paper. The $ p $-vanishing moments and discrete $ p $-vanishing moments are defined on a $ p $-adic field and the relation between them is investigated. The $ p $-vanishing moments of Haar-type and non-Haar type wavelet functions are computed. Also, the connection between $ p $-vanishing moment of non-Haar type wavelet functions and the approximation order of the indicator function of the compact open subgroup $ B_{0}(0) $\, of $ \mathbb{Q}_{p} $\, are established. A characterization of the $ p $-vanishing moments of nonorthogonal wavelets is done.
\end{abstract}

\textbf{keywords:} $ p $-adic field; pseudo-differential operator; refinable function; p-adic wavelets; approximation order; vanishing moment; compact support. \\
\textbf{AMS Subject Classification 2020:} 11E95, 11F85, 41A10, 42C40, 43A70.

\section{Introduction}
Over the past few years, wavelets have gained immense popularity in almost all fields of science and technology due to their unique time-frequency localization feature. One of the most significant facts is that, in addition to the canonical tool of representing a function by its Fourier series, there is a different representation using wavelets which is more suitable to certain problems in data compression and signal processing. The first wavelet construction is due to Alfr\'{e}d Haar in 1910. Then many mathematicians including Jean Morlet, Alex Grossmann, Yves Meyer, Stephane Mallat and Ingrid Daubechies have contributed various kinds of wavelets to theoretical and applied science. The concept of wavelet is then extended to the Euclidean space as well as many other topological spaces such as $ p $-adic fields. \par 
Wavelets can be generated from scaling functions as well as wavelet sets. In their paper, Albeverio \textit{et al.}\cite{aes1,aes2} proposed a complete characterization of scaling functions and explained what types of scaling functions form $ p $-adic multiresolution analysis. Khrennikov \textit{et al.} \cite{kss2} described the procedure to construct $ p $-adic wavelets from $ p $-adic scaling functions associated with an expansive automorphism. The authors \cite{kss1} discussed about all compactly supported orthogonal wavelet bases for $ L_{2}(\mathbb{Q}_{p}) $\, generated by the unique p-adic multiresolution analysis, i.e.,the Haar bases of $ L_{2}(\mathbb{Q}_{p}) $. \par 
In 2008, Shelkovich and Skopina \cite{ss} constructed infinitely many different multidimensional $ p $-adic Haar orthonormal wavelet bases for $ L_{2}(\mathbb{Q}_{p}^{n}) $\, and in 2009, Khrennikov and Shelkovich \cite{ks} developed infinite family of compactly supported non-Haar type $ p $-adic wavelet bases for $ L_{2}(\mathbb{Q}_{p}^{n}),\, n\geq 1 $. Kozyrev \textit{et al.} \cite{kks} also discussed about the one-dimensional and multi-dimensional wavelet bases and their relation to the spectral theory of pseudo-differential operators. \par
Reﬁnable functions play an important role in the construction and properties of wavelets. Basically, most of the wavelets are generated from reﬁnable functions. In their paper Athira and Lineesh \cite{al} discussed about the approximation order of shift-invariant space of a reﬁnable function on p-adic ﬁeld. The relation between the approximation order, accuracy of refinable function and order of the Strang–Fix condition are also established. \par  
In order to fulfill the requirements, some properties are relevant for the wavelet bases. One of the important property among them is the vanishing moment. Vanishing moments are essential in the context of compression of a signal. A vanishing moment limits the wavelet's ability to represent polynomial behaviour or information in a signal. Higher vanishing moments of wavelets are required for signal compression and denoising. In 1999, S. Mallat \cite{ms} established that the number of vanishing moments of a wavelet and the approximation order of the corresponding scaling function are equivalent in $ L_{2}(\mathbb{R}) $. Then Di-Rong Chen \textit{et al.} \cite{chr} obtained the the relation between the order of sum rules and the number of vanishing moments of wavelets on $ L_{2}(\mathbb{R}^{n}) $. An explicit formula for refinement masks providing vanishing moments is explored by Skopina\cite{sm}. In 2010, Yu Liu and L. Peng \cite{lp} proved the connection between discrete vanishing moments and sum rules on the Heisenberg group. \par 

Our goal is to extend the concept of vanishing moments to the $ p $-adic field $ \mathbb{Q}_{p} $. Section 2 contains preliminary informations about scaling functions, wavelet functions and accuracy of scaling functions on $ \mathbb{Q}_{p} $. In section 3, published results about the relationship between accuracy and vanishing moments on Euclidean spaces are discussed. The definitions of $ p $-vanishing moments of compactly supported functions on $ \mathbb{Q}_{p} $\, and discrete $ p $-vanishing moments of finitely supported sequences on $ I_{p} $\, are given in section 4. The relationship between the $ p $-vanishing moments and discrete $ p $-vanishing moments are also established in this section. In section 5, the $ p $-vanishing moments of Haar-type wavelet functions are calculated. The $ p $-vanishing moments of non-Haar type wavelet functions are computed in section 6. In this section, we proved the connection between $ p $-vanishing moment of non-Haar type wavelet functions and the approximation order of the indicator function of the compact open subgroup $ B_{0}(0) $\, of $ \mathbb{Q}_{p} $. Finally in section 7, we characterized the $ p $-vanishing moments of nonorthogonal wavelets.

\section{Preliminary}

Let $ p $\, be a prime number. Consider the completion field of $ \mathbb{Q} $\, with respect to the norm  $ \lvert \cdot  \rvert_{p} $\, defined by,
$$ \lvert x \rvert_{p} = \begin{cases}
                          0 & ;\, x = 0 \\
                          p^{-\gamma} & ;\,  x \neq 0, x= (p^{\gamma})\frac{m}{n},
                         \end{cases}
 $$
 where $ \gamma = \gamma(x) \in \mathbb{Z}, \, m,n \in \mathbb{Z} $\, not divisible by $ p $. Denote the above field as 
 $ G = \mathbb{Q}_{p} $. The canonical form of $ x \in \mathbb{Q}_{p}$, $ x \neq 0 $\, is, 
 \begin{equation}\label{eq1.1}
  x=p^{\gamma}(x_{0}+x_{1}p+x_{2}p^{2}+\cdots ),
 \end{equation}
where $ \gamma \in \mathbb{Z}, \, x_{j} \in \{ 0,1,\ldots, p-1\}, \, x_{0} \neq 0 $. Then for this $ x \in \mathbb{Q}_{p} $, 
the fractional part of $ x $\, is, 
$$ \{ x \}_{p} =\begin{cases}
		  0 & ;\, \gamma(x) \geq 0 \text { or } x = 0 \\
		  p^{\gamma}(x_{0}+x_{1}p+\cdots + x_{-\gamma-1}p^{-\gamma-1}) & ;\, \gamma(x) <0.
		\end{cases}
 $$ 
 \par 
 The dual group of $ \mathbb{Q}_{p} $\, is $ \mathbb{Q}_{p} $\, itself and the character on $ \mathbb{Q}_{p} $\, is defined as,
 \begin{equation}\label{eq1.5}
  \chi(x, \xi) = e^{2\pi i \{x\xi\}_{p}},
 \end{equation}
  where $ \{\cdot \}_{p} $\, is the fractional part of a number. Denote 
 $$ B_{\gamma}(a) = \{ x \in \mathbb{Q}_{p} : \lvert x-a \rvert_{p} \leq p^{\gamma} \}. $$ Then $ B_{0}(0) $\, is a compact 
 open subgroup of $ \mathbb{Q}_{p} $. Let $ \mu $\, be the Haar measure on $ \mathbb{Q}_{p} $\, with $ \mu(B_{0}(0)) = 1 $. 
 Denote $ d\mu(x) $\, by $ dx $. Let $ L_{q}(\mathbb{Q}_{p}) $\, be the collection of all integrable functions 
 $ f: \mathbb{Q}_{p} \rightarrow \mathbb{C} $\, such that $ \int_{\mathbb{Q}_{p}} \lvert f(x)\rvert^{q} dx < \infty $. 
 The Fourier transform of a complex-valued function $ f $\, defined on $ \mathbb{Q}_{p} $\, is defined as 
 $$ \widehat{f}(\xi) := \int_{\mathbb{Q}_{p}} f(x) \chi(x, \xi) dx. $$ \par 
 If $ E $\, is a measurable subset of $ \mathbb{Q}_{p} $\, and $ 1 \leq q< \infty $, then $ \lVert f \rVert_{q}(E) $\, denotes 
 the quantity $ (\int_{E} \lvert f(x) \rvert^{q} dx )^{1/q} $. If $ q = \infty $, then $ \lVert f \rVert_{\infty}(E) $\, 
 denotes the essential supremum of $ f $\, over $ E $. \par 
 For the $ p $-adic analysis related to the mapping $ \mathbb{Q}_{p} \rightarrow \mathbb{C} $, the operation of
differentiation is not defined. An analogy of the differentiation operator is a
pseudo-differential operator. The pseudo-differential operator $ D^{\alpha} : \phi \rightarrow D^{\alpha}\phi $\, is defined 
\cite{vvz} by
\begin{equation}\label{eq1.2}
 D^{\alpha}\phi = f_{-\alpha} * \phi 
\end{equation}
where $ f_{\alpha}(x) = \frac{\lvert x \rvert_{p}^{\alpha -1}}{\Gamma_{p}(\alpha)} $\, with $ \Gamma_{p}(\alpha) = 
\frac{1-p^{\alpha-1}}{1-p^{-\alpha}} $.
\begin{remark}\cite{vvz}
 The derivative $ D^{\alpha}\phi $, $\alpha >0 $\, is given by the expression 
 \begin{equation}\label{eq1.15}
  (D^{\alpha}\phi)(x) = \int_{\mathbb{Q}_{p}} \lvert \xi \rvert_{p} \widehat{\phi}(\xi) \chi(x,-\xi) d\xi.
 \end{equation}
\end{remark}
\begin{remark}\cite{vvz}
  For $ \alpha \in \mathbb{R} $\, and $ a \in \mathbb{Q}_{p} \setminus \{0\} $, $ D^{\alpha}\chi(a,x) = \lvert a 
  \rvert_{p}^{\alpha} \chi(a,x) $.
 \end{remark}
 \begin{remark}\cite{vvz}
 For $ \alpha \in \mathbb{R}, \, \gamma \in \mathbb{Z} $, let $ \Phi(x) = \delta(\lvert \xi \rvert_{p} - 
  p^{\gamma})f(\xi) $. Then $$ D^{\alpha}(\widehat{\Phi(x)}) = p^{\gamma \alpha} \widehat{\Phi(x)}. $$
\end{remark}
A polynomial $ r(x) $\, on $ \mathbb{Q}_{p} $\, is given by
$$ r(x) = c_{0} + c_{1} x+ \cdots + c_{n}x^{n}, \quad x \in \mathbb{Q}_{p}, $$
where each $ c_{j} \in \mathbb{Q}_{p} $. If $ c_{n} \neq 0 $, then the degree of $ r(x) $\, is $ n $. For a nonnegative integer $ k $, we denote by 
$ \mathscr{P}_{k} $\, the linear span of $ \{r(x) : \text{degree of } r(x) \leq k \} $. Then $ \mathscr{P} = \bigcup_{k=0}^{\infty} 
\mathscr{P}_{k} $\, is the linear space of all polynomials. \par 
Let us consider the set 
$$ I_{p} = \{ a = p^{-\gamma}(a_{0}+a_{1}p+ \cdots + a_{\gamma-1}p^{\gamma-1}) : \gamma \in \mathbb{N},\, a_{j} \in 
\{ 0,1,\ldots, p-1 \} \}. $$
Then there is a ``natural'' decomposition of $ \mathbb{Q}_{p} $\, into a union of mutually disjoint discs:
$ \mathbb{Q}_{p} = \bigcup_{a \in I_{p}} B_{0}(a) $. So, $ I_{p} $\, is a ``natural set'' of shifts for $ \mathbb{Q}_{p} $. \par 
We denote by $ l(I_{p}) $\, the linear space of all sequences on $ I_{p} $, and by $ l_{0}(I_{p}) $\, the linear space of all 
finitely supported sequences on $ I_{p} $. \par 
For a compactly supported function $ \phi $\, on $ \mathbb{Q}_{p} $\, and a sequence $ b \in l(I_{p}) $, the semi-convolution 
of $ \phi $\, with $ b $\, is defined by
\begin{equation}\label{eq1.3}
 \phi *' b(x) := \sum_{\alpha \in I_{p}} \phi(x - \alpha) b(\alpha).
\end{equation}
Let $ S(\phi) $\, denote the linear space $ \{\phi *'b : b \in l(I_{p}) \} $. We call $ S(\phi) $\, the shift-invariant space 
generated by $ b $. \par 
Define $ A:\mathbb{Q}_{p} \rightarrow \mathbb{Q}_{p} $\, by $ A(x) = \frac{1}{p}x $. Then, $ A $\, is an expansive 
automorphism with modulus, $ \lvert A \rvert = \mu(A(B_{0}(0))) = \mu(B_{1}(0)) = p $\, and $ A^{*} = A $.
\begin{definition}\cite{aes2} 
   A collection of closed spaces $ V_{j} \subset L_{2}(\mathbb{Q}_{p}), \, j \in \mathbb{Z} $, is called a \MRA (MRA) in 
   $ L_{2}(\mathbb{Q}_{p}) $\, if the following axioms hold: 
   \begin{enumerate}
    \item $ V_{j} \subset V_{j+1} $\, for all $ j \in \mathbb{Z} $;
    \item $ \bigcup_{j=-\infty}^{\infty} V_{j} $\, is dense in $ L_{2}(\mathbb{Q}_{p}) $;
    \item $ \bigcap_{j=-\infty}^{\infty} V_{j} = \{ 0 \} $;
    \item $ f(\cdot) \in V_{j} \Leftrightarrow f(A\cdot) \in V_{j+1} $\, for all $ j \in \mathbb{Z} $;
    \item there exists a function $ \phi \in V_{0} $\, such that $ \{\phi(\cdot -a),\, a \in I_{p}\} $\, is an orthonormal 
    basis for $ V_{0} $.
   \end{enumerate} 
  \end{definition}
 The function $ \phi $\, from axiom $ (5) $\, is called scaling function. Then we says that a MRA is generated by its scaling function $ \phi $\, (or $ \phi $\, generates the MRA). It follows immediately from axioms $ (4) $\, and $ (5) $\, that the functions $ p^{j/2}\phi(p^{-j}\cdot -a),\, a \in I_{p} $, form an orthonormal basis for $ V_{j}, \, j \in \mathbb{Z} $. Let $ \phi $\, be an orthogonal scaling function for a MRA $ \{ V_{j}\}_{j \in \mathbb{Z}} $, then 
\begin{equation}\label{eq1.4}
 \phi(x) = \sum_{a \in I_{p}} \alpha(a) \phi(p^{-1}x -a), \quad \alpha(a) \in \mathbb{C}.
\end{equation}
Such equations are called refinement equations, and their solutions are called refinable functions. \par 
Generally, a refinement equation \eqref{eq1.4} does not imply the inclusion property $ V_{0} \subset V_{1} $\, because the set 
of shifts $ I_{p} $\, does not form a group. Indeed, we need all the functions $ \phi(\cdot - b), b \in I_{p} $, to belong to 
the space $ V_{1} $, i.e., the identities $ \phi(x - b) = \sum_{a\in I_{p}} \alpha(a,b) \phi(p^{-1}x - a) $\, should be 
fulfilled for all $ b \in I_{p} $. Since $ p^{-1} b + a $\, is not in $ I_{p} $\, in general, we cannot argue that 
$ \phi(x - b) $\, belongs to $ V_{1} $\, for all $ b \in I_{p} $. Thus the refinement equation must be redefined as in the following definition.
\begin{definition}\cite{aes1,aes2}
 If $ \phi \in L_{2}(\mathbb{Q}_{p}) $\, is a refinable function and $ \text{supp}(\phi) \subset B_{N}(0) $, $ N \geq 0 $, then its 
 refinement equation is
 \begin{equation}\label{eq1.6}
  \phi(x) = \sum_{k=0}^{p^{N+1}-1} h\left(\frac{k}{p^{N+1}}\right) \phi \left( \frac{x}{p}-\frac{k}{p^{N+1}}\right), \quad \forall x \in \mathbb{Q}_{p},
 \end{equation}
with $ h \in l_{0}(I_{p}) $\, such that $ h(a) = 0 $\, for all $ a \in I_{p} \setminus \{k/p^{N+1}: k=0,1,\ldots, p^{N+1}-1 \} $ and 
\begin{equation}\label{eq1.7}
 \sum_{k=0}^{p^{N+1}-1} h\left(\frac{k}{p^{N+1}}\right) = p.
\end{equation}
\end{definition}
Taking Fourier transform on both sides of \eqref{eq1.6}, we obtain
\begin{equation}\label{eq1.8}
 \widehat{\phi}(\xi) = H(A^{-1}\xi)\widehat{\phi}(A^{-1}\xi), \quad \xi \in \mathbb{Q}_{p},
\end{equation}
where
\begin{equation}\label{eq1.9}
 H(\xi) = \frac{1}{p}\sum_{k=0}^{p^{N+1}-1} h\left(\frac{k}{p^{N+1}}\right)\chi\left(\frac{k}{p^{N+1}}, \xi\right), \quad \xi \in \mathbb{Q}_{p}.
\end{equation}

Denote by $ \mathscr{D}_{N}^{M} $\, the set of all $ p^{M} $-periodic functions supported on $ B_{N}(0) $. Then we have the following results about refinable functions \cite{aes2}.
\begin{theorem}\label{th1.1}
 A function $ \phi \in \mathscr{D}_{N}^{M}, \, M,N \geq 0 $, with $ \widehat{\phi}(0) \neq 0 $\, generates a MRA if and only if
 \begin{enumerate}
  \item $ \phi $\, is refinable;
  \item there exists at most $ p^{N} $\, integers $ l $\, such that $ 0 \leq l < p^{M+N} $\, and 
        $ \widehat{\phi}(\frac{l}{p^{M}}) \neq 0 $.
 \end{enumerate}
\end{theorem}
\begin{theorem}\label{th1.2}
 Let $ \widehat{\phi} $\, be defined by \eqref{eq1.8}, where $ H $\, is the trigonometric polynomial \eqref{eq1.9} with 
 $ H(0) = 1 $. If $ H(k) = 0 $\, for all $ k = 1,\ldots,p^{N+1}-1 $\, not divisible by $ p $, then $ \phi \in \mathscr{D}_{N}^{0} $. 
 If furthermore, $ \lvert H(k) \rvert = 1 $\, for all $ k = 1,\ldots,p^{N+1}-1 $\, divisible by $ p $, then 
 $ \{\phi(x -a),\, a \in I_{p}\} $\, is an orthonormal system. Conversely, if $ \text{supp }\widehat{\phi} \subset B_{0}(0) $\, 
 and the system $ \{\phi(x -a),\, a \in I_{p}\} $\, is orthonormal, then $ \lvert H(k) \rvert = 0 $\, whenever $ k $\, is 
 not divisible by $ p $, $ \lvert H(k) \rvert = 1 $\, whenever $ k $\, is divisible by $ p $, $ k = 1,\ldots,p^{N+1}-1 $, and 
 $ \lvert \widehat{\phi}(x) \rvert = 1 $\, for any $ x \in B_{0}(0) $.
\end{theorem}
\begin{theorem}\label{th1.3}\cite{aes2}
 There exists a unique MRA generated by an orthogonal scaling function.
\end{theorem}
\begin{remark}
 That is all orthogonal scaling functions generate the same Haar MRA.
\end{remark}

Suppose we have a p-adic MRA generated using a scaling function $ \phi $\, satisfying the refinement equation \eqref{eq1.6}, then 
the wavelet functions $ \psi_{j}, \, j = 1,\ldots, p-1 $\, are in the form
\begin{equation}\label{eq1.10}
 \psi_{j}(x) = \sum_{k=0}^{p^{N+1}-1} h_{j}\left(\frac{k}{p^{N+1}}\right) \phi \left( \frac{x}{p}-\frac{k}{p^{N+1}}\right), \quad 
 \forall x \in \mathbb{Q}_{p},
\end{equation}
where the coefficients $ h_{j}(\frac{k}{p^{N+1}}) $\, are chosen such that
\begin{equation}\label{eq1.11}
 <\psi_{j},\phi(\cdot-a)> = 0, \, \, <\psi_{j}, \psi_{k}(\cdot - a)> = \delta_{j,k} \delta_{0,a}, \,\, j,k = 1,\ldots, p-1, 
\end{equation}
for any $ a \in I_{p} $. \newline 
Set $$ h = \frac{1}{\sqrt{p}}\left(h(0),\ldots,h\left(\frac{p^{N+1}-1}{p^{N+1}}\right)\right),$$ 
$$ h_{j} = \frac{1}{\sqrt{p}}\left(h_{j}(0),\ldots,h_{j}\left(\frac{p^{N+1}-1}{p^{N+1}}\right)\right),\, j = 1,\ldots, p-1, $$
and
$$ S= \begin{bmatrix}
      0 & 0 & \cdots & 0 & 1 \\
      1 & 0 & \cdots & 0 & 0 \\
      0 & 1 & \cdots & 0 & 0 \\
      \cdots & \cdots & \cdots & \cdots \\
      0 & 0 & \cdots & 1 & 0 
     \end{bmatrix}. $$ 
In order to satisfy \eqref{eq1.11}, we need to find $ h_{j}, \, j = 1, \ldots, p-1 $\, such that the matrix
\begin{equation}\label{eq1.12}
 U = (S^{0}h,\ldots,S^{p^{N}-1}h,S^{0}h_{1},\ldots,S^{p^{N}-1}h_{1}, \ldots, S^{0}h_{p-1},\ldots,S^{p^{N}-1}h_{p-1})
\end{equation}
is unitary.\cite{kss2} \par 
Let $ \phi $\, be a compactly supported function in $ L_{q}(\mathbb{Q}_{p}) (1 \leq q \leq \infty) $. The norm in 
$ L_{q}(\mathbb{Q}_{p}) $\, is denoted by $ \lVert \cdot \rVert_{q} $. For an element $ f \in L_{q}(\mathbb{Q}_{p}) $\, and 
a subset $ G $\, of $ L_{q}(\mathbb{Q}_{p}) $, the distance from $ f $\, to $ F $, denoted by $ \text{dist}_{q}(f, F) $, 
is defined by $$ \text{dist}_{q}(f,F) := \inf_{g \in F}\lVert f-g \rVert_{q}. $$ 
Let $ S:=S(\phi)\cap L_{q}(\mathbb{Q}_{p}) $. For $ n \in \mathbb{Z} $, let $ S^{n} := {g (A_{n}^{-1} \cdot) : g \in S} $, 
where $ A_{n}(x) = p^{n}x, \, \, x \in \mathbb{Q}_{p} $. For a real number $ k \geq 0 $, we say that $ S(\phi) $\, 
provides approximation order $ k $\, if for each sufficiently smooth function $ f \in L_{q}(\mathbb{Q}_{p}) $, there exists a 
constant $ C > 0 $\, such that $$ \text{dist}_{q}(f,S^{n}) \leq C p^{-nk}, \quad \forall n \in \mathbb{N}\cup\{0\}. $$
The relation between the approximation order provided by $ S(\phi) $, the order of the Strang–Fix condition and the accuracy of $ \phi $\, are established in the following theorems. 
\begin{theorem}\label{th1.4}\cite{al}
 Let $ 1 \leq q \leq \infty $, and let $ \phi $\, be a compactly supported function in $ L_{q}(\mathbb{Q}_{p}) $\, with 
 $ \widehat{\phi}(0) \neq 0 $. For every positive integer $ k $, the following statements are equivalent:
 \begin{enumerate}
  \item $ S(\phi) $\, provides approximation order $ k $.
  \item $ S(\phi) $\, contains $\mathscr{P}_{k-1} $.
  \item $ D^{\mu}\widehat{\phi}(p^{N}\alpha) = 0, \, \forall \alpha \in I_{p}\setminus \{0\},\, \lvert \mu \rvert \leq k-1 $.
 \end{enumerate}
\end{theorem}
\begin{definition}\cite{al}
For a compactly supported function $ \phi $\, on $ \mathbb{Q}_{p} $, define 
$$ N(\phi) = \{ \xi \in \mathbb{Q}_{p} : \widehat{\phi}(\xi) = 0 \}.$$
\end{definition}
\begin{theorem}\label{th1.5}\cite{al}
 Let $ h $\, be a finitely supported sequence on $ I_{p} $, satisfying \eqref{eq1.7} and $ H $\, be a function given by \eqref{eq1.9}. Let $ \Omega = \{ 0,\frac{1}{p}, \frac{2}{p}, \ldots, \frac{p-1}{p} \} $. If 
 \begin{equation}\label{eq1.13}
  D^{\mu}H(p^{N}A^{-1}\omega) = 0, \, \forall \omega \in \Omega \setminus \{0\},\, \lvert \mu \rvert \leq k-1,
 \end{equation}
then the normalized solution $ \phi $\, of \eqref{eq1.6} has accuracy $ k $. Conversely, if $ \phi $\, has accuracy $ k $\, 
and if $ N(\phi)\cap(p^{N}A^{-1}\Omega) = \emptyset $, then \eqref{eq1.13} holds true.
\end{theorem}

\section{Vanishing Moments on Euclidean spaces}

In this section we recall some facts about vanishing moments on Euclidean spaces.\par
Let $ \psi $\, be a wavelet function on $ \mathbb{R} $. Then we say that $ \psi $\, has $ n $\, vanishing moments if 
$$ \int_{-\infty}^{\infty} t^{k}\psi(t)dt = 0, \text{  for  } 0\leq k < n. $$
A wavelet with $ n $\, vanishing moments is orthogonal to each of the polynomials of degree $ n-1 $ \cite{ms}.\par
The following theorem gives the relationship between the approximation order of the corresponding scaling function and the 
number of vanishing moments of the wavelets on $ \mathbb{R} $. 
\begin{theorem}\label{th4.2.1}[Theorem 7.4 in Chapter 7 of \cite{ms}]
 Let $ \psi $\, and $ \phi $\, be a wavelet and a scaling function that generate an orthogonal basis. Let $ h $\, be the 
 refinement mask for the scaling function $ \phi $. Suppose that $ \lvert \psi(x) \rvert = O((1+t^{2})^{-\frac{n}{2}-1}) $\, and 
 $ \lvert \phi(x) \rvert = O((1+t^{2})^{-\frac{n}{2}-1}) $. Then the following statements are equivalent:
 \begin{enumerate}
  \item The wavelet $ \psi $\, has $ n $\, vanishing moments.
  \item $ \widehat{\psi}(\omega) $\, and its first $ n-1 $\, derivatives are zero at $ \omega = 0 $.
  \item $ \widehat{h}(\omega) $\, and its first $ n-1 $\, derivatives are zero at $ \omega = \pi $.
  \item For any $ 0 \leq k <n $, 
         \begin{equation}\label{eq4.2.1}
          q_{k}(t) = \sum_{n=-\infty}^{\infty} n^{k} \phi(t-n) \text{  is a polynomial of degree } k.
         \end{equation}
 \end{enumerate}
\end{theorem}
The hypothesis $(4)$\, is called the Strang-Fix condition on $ \mathbb{R} $. \newline \par 

Let $ M $\, be a $ d \times d $\, dilation matrix on $ \mathbb{R}^{d} $\, and denote 
$ m := \lvert \det M \rvert $. Let $ \Omega $\, be a complete set of representatives of the distinct cosets of the quotient 
group $ \mathbb{Z}^{d}/M \mathbb{Z}^{d} $. It is obvious that the cardinality of $ \Omega $\, is equal to $ m $. 
Without loss of generality we assume that $ 0 \in \Omega $. \par 
For any positive integer $ k $, a compactly supported function $ f \in L_{2}(\mathbb{R}^{d}) $\, has $ k $\, vanishing 
moments if
$$ \int_{\mathbb{R}^{d}} p(x)f(x)dx = 0, \quad \text{for all } p \in \mathscr{P}_{k}, $$
where $ \mathscr{P}_{k} $\, denotes the set of all polynomials of degree less than $ k $. \par 
Let $ l_{0}(\mathbb{Z}^{d}) $\, denotes the space of all finitely supported sequences on $ \mathbb{Z}^{d} $. Then a sequence $ b \in l_{0}(\mathbb{Z}^{d}) $\, has $ k $\, discrete vanishing moments if it satisfies the following equalities
$$ \sum_{\gamma \in \mathbb{Z}^{d}} b(\gamma)\gamma^{j} = 0, \quad \text{for any } j \in \mathbb{Z}_{+}^{d}, \, \, 
\lvert j \rvert = j_{1} + \cdots + j_{d} <k, $$
where $ \gamma^{j} = \gamma_{1}^{j_{1}}\ldots \gamma_{d}^{j_{d}} $. 
\begin{proposition}\label{th4.2.2}\cite{chr}
 Given a multiresolution analsysis on $ \mathbb{R}^{d} $, with $ m-1 $\, wavelets $ \psi_{\epsilon}, \epsilon \in \Omega \setminus \{0\} $. Then $ \psi_{\epsilon}, \epsilon \in \Omega \setminus \{0\} $, have $ k $\, vanishing 
 moments if and only if the corresponding wavelet filters $ \{h_{\epsilon}(\gamma)\}, \epsilon \in \Omega \setminus \{0\} $, 
 have $ k $\, discrete vanishing moments.
\end{proposition}
A finitely supported sequence $ h $\, on $ \mathbb{Z}^{d} $\, satisfies the sum rules \cite{rqj1} of order $ k $\, if
$$ \sum_{\gamma \in \mathbb{Z}^{d}} h(\epsilon+M(\gamma))p(\epsilon+M(\gamma)) = \sum_{\gamma \in \mathbb{Z}^{d}} h(M(\gamma))
p(M(\gamma)) $$
for all $ \epsilon \in \Omega, \, p \in \mathscr{P}_{k} $. \par
The following proposition gives the relationship between the order of sum rules and the number of discrete vanishing moments 
on $ \mathbb{R}^{d} $.
\begin{proposition}\label{th4.2.3}\cite{chr}
 Let $ h \in l_{0}(\mathbb{Z}^{d}) $\, satisfies the sum rules of order $ k $\, and $ \sum_{\gamma \in \mathbb{Z}^{d}} 
 h(\gamma) = 2^{d} $. If $ b \in l_{0}(\mathbb{Z}^{d}) $\, satisfies $ \sum_{\gamma \in \mathbb{Z}^{d}} b(\gamma)
 h(\gamma-M(\beta)) = 0 $, for any $ \beta \in \mathbb{Z}^{d} $, then $ b $\, has $ k $\, discrete vanishing moments. 
\end{proposition}
If the refinement mask $ h $\, of the scaling function $ \phi $\, satisfies the sum rule of order $ k $, then the corresponding 
wavelet filters $ h_{\epsilon} $\, satisfies 
$$ \sum_{\gamma \in \mathbb{Z}^{d}} h_{\epsilon}(\gamma)h(\gamma-M(\beta)) = 0, \quad \text{for any } \beta \in \mathbb{Z}^{d}, 
\, \epsilon \in \Omega \setminus \{ 0\}. $$ 
Thus, by Propositions \ref{th4.2.2} and \ref{th4.2.3}, the corresponding wavelets $ \psi_{\epsilon} $\, has vanishing moments of order $ k $, for each $ \epsilon \in \Omega \setminus \{ 0\} $.

\section{$ p $-Vanishing Moments on $ \mathbb{Q}_{p} $}

This section gives a description of vanishing moments and discrete vanishing moments on $ \mathbb{Q}_{p} $. 
\begin{definition}\label{df4.3.1}
 For any positive integer $ k $, a compactly supported function $ f $\, in $ L_{2}(\mathbb{Q}_{p}) $\, is said to have $ k $\, 
 $ p $-vanishing moments if 
 \begin{equation}\label{eq4.3.1}
  \int_{\mathbb{Q}_{p}} \lvert x \rvert_{p}^{\mu}f(x)dx = 0, \quad \text{for all } 0\leq \mu < k .
 \end{equation}
\end{definition}
\begin{theorem}\label{th4.3.1}
 Let $ f $\, be a compactly supported function in $ L_{2}(\mathbb{Q}_{p}) $. Suppose that $ \widehat{f} $\, is $ k $\, times pseudo-differentiable. Then $ f $\, has $ k $ $ p $-vanishing moments if and only if $ D^{\mu}\widehat{f}(0) = 0, 
 \text{ for all } 0 \leq \mu < k $.
\end{theorem}
\begin{proof}
 From the definition of the pseudo-differential operator $ D^{\mu} $\, we have,
 $$ (D^{\mu} f )(x) = \int_{\mathbb{Q}_{p}} \lvert \xi \rvert_{p}^{\mu} \widehat{f}(\xi) \chi (x, -\xi)d\xi, \qquad \mu > 0. $$
 Thus we have, $$ (D^{\mu} \widehat{f} ) (\xi) = \int_{\mathbb{Q}_{p}} \lvert x \rvert_{p}^{\mu} f(x) \chi (\xi, -x)dx, \qquad \mu > 0. $$
Then, $$ (D^{\mu}\widehat{f})(0) = \int_{\mathbb{Q}_{p}} \lvert x \rvert_{p}^{\mu} f(x)dx, \qquad \mu > 0. $$
Hence, $ f $\, has $ k $ $ p $-vanishing moments if and only if $ D^{\mu}\widehat{f}(0) = 0, \, \forall \, \, 0 \leq \mu < k $.
 \end{proof}
\begin{definition}\label{df4.3.2}
 A sequence $ b \in l_{0}(I_{p}) $\, has $ k $\, discrete $ p $-vanishing moments if it satisfies the equalities
 \begin{equation}\label{eq4.3.2}
  \sum_{a \in I_{p}} b(a)\lvert a \rvert_{p}^{\mu} = 0, \quad \text{for all } 0 \leq \mu < k.
 \end{equation}
\end{definition}
\begin{theorem}\label{th4.3.2}
 Given a multiresolution analsysis on $ \mathbb{Q}_{p} $, with $ p-1 $\, wavelets $ \psi_{j}, \, j = 1, \ldots, p-1 $. Then $ \psi_{j} $, $ j = 1, \ldots, p-1 $, have $ k $\, $ p $-vanishing moments if and only if the 
 corresponding wavelet filters $ h_{j} , \, j = 1, \ldots, p-1 $, have $ k $\, discrete $ p $-vanishing moments.
\end{theorem}
\begin{proof}
 From Theorem \ref{th4.3.1},  $ \psi_{j}, \, j = 1, \ldots, p-1 $, have $ k $\, $ p $-vanishing moments if and only if 
 $ D^{\mu}\widehat{\psi_{j}}(0) = 0, \text{ for all } 0 \leq \mu < k, \, j= 1, \ldots, p-1 $. By the definition of wavelets \eqref{eq1.10}, 
 $ \psi_{j}, \, j = 1, \ldots, p-1 $\, are of the form, 
 $$ \psi_{j}(x) = \sum_{k=0}^{p^{N+1}-1} h_{j}\left(\frac{k}{p^{N+1}}\right) \phi \left( \frac{x}{p}-\frac{k}{p^{N+1}}\right), \quad x \in \mathbb{Q}_{p}. $$
 Taking the Fourier transform on both sides, 
 \begin{equation}\label{eq4.3.3}
 \widehat{\psi}_{j}(\xi) = H_{j}(T^{-1}\xi)\widehat{\phi}(T^{-1}\xi), \quad \xi \in \mathbb{Q}_{p},
\end{equation}
where
\begin{equation}\label{eq4.3.4}
 H_{j}(\xi) = \frac{1}{p}\sum_{k=0}^{p^{N+1}-1} h_{j}\left(\frac{k}{p^{N+1}}\right)\chi\left(\frac{k}{p^{N+1}}, \xi\right), \quad \xi \in \mathbb{Q}_{p},\, 
 j = 1, \ldots, p-1.
\end{equation} 
Thus, $ D^{\mu}\widehat{\psi_{j}}(0) = 0, \text{ for all } 0 \leq \mu < k, \, j= 1, \ldots, p-1 $\, if and only if 
$$ D^{\mu}H_{j}(0) = 0, \text{ for all } 0 \leq \mu < k, \, j= 1, \ldots, p-1. $$ That is, if and only if 
$$ \sum_{k=0}^{p^{N+1}-1} h_{j}\left(\frac{k}{p^{N+1}}\right)\left \lvert \frac{k}{p^{N+1}} \right \rvert_{p}^{\mu} = 0, \text{ for all } 0 \leq \mu < k, 
\, j= 1, \ldots, p-1. $$ 
i.e., if and only if each $ h_{j} , \, j = 1, \ldots, p-1 $, has $ k $\, discrete $ p $-vanishing moments.
\end{proof}

\begin{theorem}\label{th4.3.3}
 Let $ \phi \in L_{2}(\mathbb{Q}_{p}) $\, be a scaling function that generates a MRA and $ \text{supp}(\phi)\subset B_{0}(0) $. 
 Then the corresponding refinement mask of $ \phi $\, is of the form $ \left\{ h\left(\frac{k}{p}\right) \right\} $\, with $ h\left(\frac{k}{p}\right) = 1, 
 \text{ for all } k = 0,1, \ldots, p-1 $. Moreover, $ \phi $\, has accuracy 1.
\end{theorem}
\begin{proof}
 Suppose that the scaling function $ \phi $\, that generates an MRA is given by 
 $$  \phi(x) = \sum_{k=0}^{p-1} h\left(\frac{k}{p}\right) \phi \left(\frac{x}{p}-\frac{k}{p} \right). $$ 
 Then by \eqref{eq1.9}, $$ H(\xi) = \frac{1}{p} \sum_{k=0}^{p-1}h\left(\frac{k}{p}\right)\chi\left(\frac{k}{p}, \xi\right). $$
 Now by theorem \ref{th1.2}, $ H(0) = 1 $\, and $ H(k) = 0 $\, for all $ k = 1,\ldots, p-1 $. Also, we have $ \sum_{k=0}^{p-1}h\left(\frac{k}{p}\right) = p $. 
 Thus we have a system of $ p $\, equations $ Uh = B $, where 
 $$ U = \frac{1}{\sqrt{p}}\begin{bmatrix}
                           1 & 1 & 1 & \cdots & 1 \\
                           1 & \chi\left(\frac{1}{p},1\right) & \chi\left(\frac{2}{p},1\right) & \cdots & \chi\left(\frac{p-1}{p},1\right) \\
                           1 & \chi\left(\frac{1}{p},2\right) & \chi\left(\frac{2}{p},2\right) & \cdots & \chi\left(\frac{p-1}{p},2\right) \\
                           \vdots & \vdots & \vdots & \ddots & \vdots \\
                           1 & \chi\left(\frac{1}{p},p-1\right) & \chi\left(\frac{2}{p},p-1\right) & \cdots & \chi\left(\frac{p-1}{p},p-1\right) \\
                         \end{bmatrix}, $$
 $$ h = \begin{bmatrix}
     h(0) \\
     h\left(\frac{1}{p}\right) \\
     h\left(\frac{2}{p}\right) \\
     \vdots \\
     h\left(\frac{p-1}{p}\right)
    \end{bmatrix} \text{ and }
B = \begin{bmatrix}
     \sqrt{p} \\
     0 \\
     0 \\
     \vdots \\
     0
    \end{bmatrix}
$$
Since $ \chi\left(\frac{k}{p},l\right),\, k,l = 0, 1, \ldots, p-1 $\, are the $ p^{\text{th}} $\, roots of unity, the matrix $ U $\, is unitary. Also, $ U $\, is symmetric. Hence $ U^{-1} = \bar{U} $. Then $ Uh = B $\, has a solution $ h = U^{-1}B $. That is,
$$ \begin{aligned}
    h & = \begin{bmatrix}
          h(0) \\
          h\left(\frac{1}{p}\right) \\
          h\left(\frac{2}{p}\right) \\
          \vdots \\
          h\left(\frac{p-1}{p}\right)
        \end{bmatrix} \\
    & = \frac{1}{\sqrt{p}}\begin{bmatrix}
                           1 & 1 & 1 & \cdots & 1 \\
                           1 & \overline{\chi\left(\frac{1}{p},1\right)} & \overline{\chi\left(\frac{2}{p},1\right)} & \cdots & \overline{\chi\left(\frac{p-1}{p},1\right)} \\
                           1 & \overline{\chi\left(\frac{1}{p},2\right)} & \overline{\chi\left(\frac{2}{p},2\right)} & \cdots & \overline{\chi\left(\frac{p-1}{p},2\right)} \\
                           \vdots & \vdots & \vdots & \ddots & \vdots \\
                           1 & \overline{\chi\left(\frac{1}{p},p-1\right)} & \overline{\chi\left(\frac{2}{p},p-1\right)} & \cdots & \overline{\chi\left(\frac{p-1}{p},p-1\right)} \\
                         \end{bmatrix}
    \begin{bmatrix}
     \sqrt{p} \\
     0 \\
     0 \\
     \vdots \\
     0
    \end{bmatrix}
    \end{aligned}
    $$
    Hence $ h\left(\frac{k}{p}\right) = 1 $, for all $ k = 0, 1, \ldots, p-1 $. \par 
    That is, $ H(\xi) = \frac{1}{p} \sum_{k=0}^{p-1} \chi(\frac{k}{p}, \xi) $. Since $ \sum_{k=0}^{p-1} \chi(\frac{k}{p}, \xi) = 0, \text{ for all } \xi = 1,2, \ldots, p-1 $, we have $ H(j) = 0,\text{ for all } \xi = 1,2, \ldots, p-1 $\, and for $ \mu > 1 $,
    $$ D^{\mu}H(\xi) = \frac{1}{p} \sum_{k=0}^{p-1} \left \lvert \frac{k}{p} \right \rvert_{p}^{\mu} \chi(\frac{k}{p}, \xi) = p^{\mu -1} \sum_{k=1}^{p-1} \chi(\frac{k}{p}, \xi). $$
     But $ \sum_{k=1}^{p-1} \chi(\frac{k}{p}, \xi) \neq 0, \text{ for all } \xi = 1,2, \ldots, p-1 $,$ D^{\mu}H(T^{-1}\omega) \neq 0 $, for all $ \omega \in \{ \frac{1}{p}, \frac{2}{p}, \ldots, \frac{p-1}{p} \},\, \mu \geq 0 $. Thus by theorem \ref{th1.5}, the accuracy of $ \phi $\, is 1.  
\end{proof}
\begin{remark}
 This scaling function $ \phi $\, is called the Haar scaling function.
\end{remark} 
\begin{theorem}\label{th4.3.4}
 Let $ \phi $\, be a Haar scaling function and $ \psi_{j}, \, j= 1,\ldots, p-1 $\, be the corresponding wavelets on $ \mathbb{Q}_{p} $. Then $ \psi_{j} $\, has only $ 1 $\, $ p $-vanishing moment for each $ j= 1,\ldots, p-1 $.
\end{theorem}
\begin{proof}
 Here $ \psi_{j}(x) = \sum_{k=0}^{p-1} h_{j}\left(\frac{k}{p}\right) \phi\left(\frac{x}{p}-\frac{k}{p}\right) $, where $ h_{j} $\, are chosen so that the matrix $ U $\, in \eqref{eq1.12} is unitary. That is $ h_{j} $\, satisfies the condition
 \begin{equation}\label{eq4.3.5}
  \sum_{k=0}^{p-1} h\left(\frac{k}{p}\right)\overline{h_{j}\left(\frac{k}{p}\right)} = 0, \, \text{ for all } j = 1, \ldots, p-1.
 \end{equation}
 Since $ \phi $\, is the Haar scaling function, $ h\left(\frac{k}{p}\right) = 1 $, for all $ k=0,1,\ldots, p-1 $. Thus from \eqref{eq4.3.5}, we can conclude that $$ \sum_{k=0}^{p-1}h_{j}\left(\frac{k}{p}\right) = 0,\, \forall j =1,\ldots, p-1. $$ 
 \par 
Suppose that $ \psi_{j}, \, j= 1, \ldots, p-1 $\, have $ k $\, $ p $-vanishing moments for $ k>1 $. Then $ \sum_{k=0}^{p-1} h_{j}\left(\frac{k}{p}\right)\left \lvert \frac{k}{p} \right\rvert_{p}^{\mu} = 0,\, 0\leq \mu <k $. That is, $$ \sum_{k=0}^{p-1} h_{j}\left(\frac{k}{p}\right) = 0 \text{  and  }   p^{\mu} \sum_{k=1}^{p-1} h_{j}\left(\frac{k}{p}\right) = 0, \, \forall j = 1, \ldots, p-1,\, 0<\mu<k, $$
which is possible only when $ h_{j}(0) = 0 $, for all $ j =1, \ldots, p-1 $. \par 
Since the matrix $ U $\, in \eqref{eq1.12} is unitary, its rows are orthonormal. That is $ \lvert h(0) \rvert^{2}+ \sum_{j=1}^{p-1} \lvert h_{j}(0) \rvert^{2} = p $. This is a contradiction, since $ h_{j}(0) = 0 $, for all $ j =1, \ldots, p-1 $. \par 
Hence $ \psi_{j} $\, has only $ 1 $\, $ p $-vanishing moment for each $ j= 1,\ldots, p-1 $.
\end{proof}

 \begin{remark}
  \begin{enumerate}
   \item If $ \phi $\, be an orthogonal scaling function that generates a MRA and $ \text{supp}(\phi)\subset B_{N}(0),\, N>0 $. Then, by theorem \ref{th1.2}, $ \phi $\, doesn't satisfies the sum rules and $ \phi $\, has no accuracy.
   \item Let $ \phi $\, be a scaling function that generates a MRA and $ \text{supp}(\phi)\subset B_{0}(0) $. By theorems \ref{th4.3.3} and \ref{th4.3.4}, if $ \phi $\, has accuracy 1 then the corresponding wavelet has 1 $ p $-vanishing moment. But this cannot be true for scaling functions having support $ \subset B_{N}(0),\, N>0 $.
  \end{enumerate}
 \end{remark}

\begin{example}
 Let $ p=3,\, N=1 $. Set $ H(k) = 0 $\, if $ k $\, is not divisible by $ 3 $, and $ H(0) = 1,\, H(3) = H(6) = -1 $. Then 
 $$ \widehat{\phi}(\xi) = \begin{cases}
                           1, & \lvert \xi \rvert_{3} \leq \frac{1}{3}, \\
                           -1, & \lvert \xi -1 \rvert_{3} \leq \frac{1}{3}, \\
                           -1, & \lvert \xi -2 \rvert_{3} \leq \frac{1}{3}, \\
                           0, & \lvert \xi \rvert_{3} \geq 3.
                          \end{cases}
$$
Thus, the corresponding refinement equation is $ \phi(x) = \sum_{k=0}^{8} h(k/9) \phi(x/3-k/9) $\, where $ h(0) = h(3/9) = h(6/9) = -1/3 $, 
$ h(1/9) = h(2/9) = h(4/9) = h(5/9) = h(7/9) = h(8/9) = 2/3 $. That is 
\begin{equation}
 \phi(x) = \begin{cases}
                           \frac{-1}{3}, & \lvert x \rvert_{3} \leq 1, \\
                           \frac{2}{3}, & \lvert x -\frac{1}{3} \rvert_{3} \leq 1, \\
                           \frac{2}{3}, & \lvert x -\frac{2}{3} \rvert_{3} \leq 1, \\
                           0, & \lvert \xi \rvert_{3} \geq 9.
                          \end{cases}
\end{equation}
The corresponding wavelet functions are
$$ \psi_{1} = \sqrt{\frac{3}{2}}\left(\phi\left(\frac{x}{3}\right)-\phi\left(\frac{x}{3}-\frac{1}{3}\right)\right) $$ and
$$ \psi_{2} = \frac{1}{\sqrt{2}}\left(\phi\left(\frac{x}{3}\right)+\phi\left(\frac{x}{3}-\frac{1}{3}\right)-2\phi\left(\frac{x}{3}-\frac{2}{3}\right)\right). $$
That is 
$$ \psi_{1}(x) = \begin{cases}
                           -\sqrt{\frac{3}{2}}, & \lvert x \rvert_{3} \leq \frac{1}{3}, \\
                           \sqrt{\frac{3}{2}}, & \lvert x -1 \rvert_{3} \leq \frac{1}{3}, \\
                           -0, & \lvert x -2 \rvert_{3} \leq \frac{1}{3}, \\
                           0, & \lvert x \rvert_{3} \geq 3 \\
                          \end{cases} \text{ and }
\psi_{2}(x) = \begin{cases}
                           -\frac{1}{\sqrt{2}}, & \lvert x \rvert_{3} \leq \frac{1}{3}, \\
                           -\frac{1}{\sqrt{2}}, & \lvert x -1 \rvert_{3} \leq \frac{1}{3}, \\
                           \sqrt{2}, & \lvert x -2 \rvert_{3} \leq \frac{1}{3}, \\
                           0, & \lvert x \rvert_{3} \geq 3 \\
                          \end{cases}.
$$
Here $ \phi $\, has no accuracy. But $ \psi_{1} $\, and $ \psi_{2} $\, have $ 1 $\, $ p $-vanishing moment.
\end{example}

\section{\texorpdfstring{$ p $}p-Vanishing Moments of Haar-type Wavelets}

In contrast to the real case, the wavelet basis generated by the Haar MRA in the $ p $-adic case is not unique. In \cite{kss1}, the authors provides an explicit description about the family of wavelet functions generated by the Haar MRA in $ L_{2}(\mathbb{Q}_{p}) $.
\begin{theorem}\label{th4.4.1}\cite{kss1}
 Let $ \phi = 1_{B_{0}(0)} $\, be the indicator function of $ B_{0}(0) $\, and $$ \psi_{\nu}^{(0)}(x) = 
 \sum_{r=0}^{p-1} e^{2\pi i\nu r/p}\phi\left(\frac{x}{p}-\frac{r}{p}\right), \, x \in \mathbb{Q}_{p},\, \nu = 1,\ldots, p-1. $$ Then the set of all compactly supported wavelet functions on $ \mathbb{Q}_{p} $\, are given by 
 \begin{equation}\label{eq4.4.1}
  \psi_{j}(x) = \sum_{\nu=1}^{p-1} \sum_{k=0}^{p^{s}-1} \alpha_{\nu,k}^{j} \psi_{\nu}^{(0)}\left(x-\frac{k}{p^{s}}\right), \, x \in \mathbb{Q}_{p},\, j = 1,\ldots, p-1,
 \end{equation}
where $ s= 0,1,2, \ldots $, and 
\begin{equation}\label{eq4.4.2}
 \alpha_{\nu,k}^{j} = \begin{cases}
                         -p^{-s} \sum_{m=0}^{p^{s}-1} e^{-2\pi i \frac{\frac{-\nu}{p}+m}{p^{s}}k} \sigma_{j m} z_{j j}, & \text{ if } j = \nu , \\
                         p^{-2s} \sum_{m=0}^{p^{s}-1} \sum_{n=0}^{p^{s}-1}e^{-2\pi i \frac{\frac{-\nu}{p}+m}{p^{s}}k} \left(\frac{1-e^{ 2 \pi i\frac{j -\nu}{p}}}{e^{ 2\pi i \frac{\frac{j - \nu}{p}+m-n}{p^{s}}} -1}\right) \sigma_{\nu m} z_{\nu j}, & \text{ if } j \neq \nu,
                        \end{cases}
\end{equation}
$ \lvert \sigma_{j m} \rvert = 1 $\, and $ z_{j \nu} $\, are entries of a $ (p-1) \times (p-1) $\, unitary matrix $ Z $.
\end{theorem}

\begin{example}
 Let $ s= 2 $\, and $ p =2 $. Then by \eqref{eq4.4.1}, the wavelet function is given by
 $$ \psi_{1}(x) = \sum_{k=0}^{3} \alpha_{1,k}^{1} \psi_{1}^{(0)}\left(x-\frac{k}{4}\right), \, x \in \mathbb{Q}_{p}. $$
 That is  \begin{equation*}
            \begin{split}
            \psi_{1}(x) = \, & \alpha_{1,0}^{1} \left(\phi\left(\frac{x}{2}\right)-\phi\left(\frac{x}{2}-\frac{4}{8}\right)\right) + \alpha_{1,1}^{1} \left(\phi\left(\frac{x}{2}-\frac{1}{8}\right)-\phi\left(\frac{x}{2}-\frac{5}{8}\right)\right) \\
             & + \alpha_{1,2}^{1} \left(\phi\left(\frac{x}{2}-\frac{2}{8}\right)-\phi\left(\frac{x}{2}-\frac{6}{8}\right)\right) + \alpha_{1,3}^{1} \left(\phi\left(\frac{x}{2}-\frac{3}{8}\right)-\phi\left(\frac{x}{2}-\frac{7}{8}\right)\right).
            \end{split}
          \end{equation*}
Here $ h_{1}(0) = \alpha_{1,0}^{1},\,  h_{1}(1/8) = \alpha_{1,1}^{1},\, h_{1}(2/8) = \alpha_{1,2}^{1},\,  h_{1}(3/8) = \alpha_{1,3}^{1},\, h_{1}(4/8) = -\alpha_{1,0}^{1},\,  h_{1}(5/8) = -\alpha_{1,1}^{1},\, h_{1}(6/8) = -\alpha_{1,2}^{1},\,  h_{1}(7/8) = -\alpha_{1,3}^{1} $. \par 
Then we have $ \sum_{k=0}^{7} h_{1}(k/8) = 0 $. That is $ \psi_{1} $\, has $ 1 $\, $ p $-vanishing moment. Also, 
\begin{equation*}
 \begin{aligned}
  \sum_{k=0}^{7} \left \lvert \frac{k}{8} \right\rvert_{2}^{\mu} h_{1}\left(\frac{k}{8}\right) = & 2^{3\mu}\left(h_{1}\left(\frac{1}{8}\right) + h_{1}\left(\frac{3}{8}\right) + h_{1}\left(\frac{5}{8}\right) + h_{1}\left(\frac{7}{8}\right)\right) \\ & + 2^{2\mu}\left(h_{1}\left(\frac{2}{8}\right)+h_{1}\left(\frac{6}{8}\right)\right)+2^{\mu} h_{1}\left(\frac{4}{8}\right) \\
  = & 8^{\mu}(\alpha_{1,1}^{1} + \alpha_{1,3}^{1} - \alpha_{1,1}^{1} - \alpha_{1,3}^{1}) + 4^{\mu} (\alpha_{1,2}^{1}-\alpha_{1,2}^{1}) - 2^{\mu} \alpha_{1,0}^{1} \\
  = & -2^{\mu} \alpha_{1,0}^{1}.
 \end{aligned}
\end{equation*}
Thus $$  \sum_{k=0}^{7} \left \lvert \frac{k}{8} \right\rvert_{2}^{\mu} h_{1}\left(\frac{k}{8}\right) = 0 \Leftrightarrow h_{1}(0) = 0,\, \mu>0. $$
That is $ \psi_{1} $\, has infinite number of $ p $-vanishing moments if and only if $ h_{1}(0) = 0 $.
\end{example}
\begin{example}
 Let $ s= 1 $\, and $ p =3 $. Then by \eqref{eq4.4.1}, the wavelet function is given by
 $$ \psi_{j}(x) = \sum_{\nu = 1}^{2} \sum_{k=0}^{2} \alpha_{\nu,k}^{j}\psi_{\nu}^{(0)}\left(x-\frac{k}{3}\right), \, x \in \mathbb{Q}_{p}, $$
 where $$ \psi_{1}^{(0)}(x) = \phi\left(\frac{x}{3}\right)+ \frac{-1+i\sqrt{3}}{2}\phi\left(\frac{x}{3}-\frac{1}{3}\right) + \frac{-1-i\sqrt{3}}{2} \phi\left(\frac{x}{3}-\frac{2}{3}\right), $$
 $$ \psi_{2}^{(0)}(x) = \phi\left(\frac{x}{3}\right)+ \frac{-1-i\sqrt{3}}{2}\phi\left(\frac{x}{3}-\frac{1}{3}\right) + \frac{-1+i\sqrt{3}}{2} \phi\left(\frac{x}{3}-\frac{2}{3}\right). $$
 That is, $$ h_{j}(0) = (\alpha_{1,0}^{j}+\alpha_{2,0}^{j}), \, h_{j}\left(\frac{1}{9}\right) = (\alpha_{1,1}^{j}+\alpha_{2,1}^{j}), \, h_{j}\left(\frac{2}{9}\right) = (\alpha_{1,2}^{j}+\alpha_{2,2}^{j}), $$
 $$ h_{j}\left(\frac{3}{9}\right) = \left(\alpha_{1,0}^{j}\frac{-1+i\sqrt{3}}{2} +\alpha_{2,0}^{j}\frac{-1-i\sqrt{3}}{2}\right) , $$
$$ h_{j}\left(\frac{4}{9}\right) = \left(\alpha_{1,1}^{j}\frac{-1+i\sqrt{3}}{2} +\alpha_{2,1}^{j} \frac{-1-i\sqrt{3}}{2}\right), $$
 $$ h_{j}\left(\frac{5}{9}\right) = \left(\alpha_{1,2}^{j}\frac{-1+i\sqrt{3}}{2} +\alpha_{2,2}^{j}\frac{-1-i\sqrt{3}}{2}\right), $$
 $$ h_{j}\left(\frac{6}{9}\right) = \left(\alpha_{1,0}^{j}\frac{-1-i\sqrt{3}}{2} +\alpha_{2,0}^{j}\frac{-1+i\sqrt{3}}{2}\right), $$
 $$ h_{j}\left(\frac{7}{9}\right) = \left(\alpha_{1,1}^{j}\frac{-1-i\sqrt{3}}{2} +\alpha_{2,1}^{j}\frac{-1+i\sqrt{3}}{2}\right), $$
 $$ h_{j}\left(\frac{8}{9}\right) = \left(\alpha_{1,2}^{j}\frac{-1-i\sqrt{3}}{2} +\alpha_{2,2}^{j}\frac{-1+i\sqrt{3}}{2}\right). $$
 Then we have $ \sum_{k=0}^{8} h_{j}(k/9) = 0 $. That is $ \psi_{j} $\, has $ 1 $\, $ p $-vanishing moment. Also, 
\begin{equation*}
 \begin{aligned}
  \sum_{k=0}^{8} \left \lvert \frac{k}{9} \right \rvert_{3}^{\mu} & h_{j}\left(\frac{k}{9}\right) \\ & = 3^{2\mu}\left(h_{j}\left(\frac{1}{9}\right) + h_{j}\left(\frac{2}{9}\right) + h_{j}\left(\frac{4}{9}\right) + h_{j}\left(\frac{5}{9}\right) + h_{j}\left(\frac{7}{9}\right)+h_{j}\left(\frac{8}{9}\right)\right)\\ & + 3^{\mu}\left( h_{j}\left(\frac{3}{9}\right) + h_{j}\left(\frac{6}{9}\right)\right) \\
  = & -3^{\mu}(\alpha_{1,0}^{j}+\alpha_{2,0}^{j}) .
 \end{aligned}
\end{equation*}
Thus $$  \sum_{k=0}^{8} \left \lvert \frac{k}{9} \right\rvert_{3}^{\mu} h_{j}\left(\frac{k}{9}\right) = 0 \Leftrightarrow h_{j}(0) = 0,\, \mu>0,\, j = 1,2. $$
That is $ \psi_{j} $\, has infinite number of $ p $-vanishing moments if and only if $ h_{j}(0) = 0 $.
\end{example}
\begin{remark}
 We can rewrite \eqref{eq4.4.1} as 
 $$  \begin{aligned}
      \psi_{j}(x) & = \sum_{\nu=1}^{p-1} \sum_{k=0}^{p^{s}-1} \alpha_{\nu,k}^{j} \psi_{\nu}^{(0)}\left(x-\frac{k}{p^{s}}\right) \\
      & = \sum_{\nu=1}^{p-1} \sum_{k=0}^{p^{s}-1} \alpha_{\nu,k}^{j} \sum_{r=0}^{p-1} e^{2\pi i\nu r/p}\phi\left(\frac{x}{p}-\frac{k}{p^{s+1}}-\frac{r}{p}\right) \\
      & = \sum_{l=0}^{p^{s+1}-1} h_{j}\left(\frac{l}{p^{s+1}}\right) \phi \left(\frac{x}{p} - \frac{l}{p^{s+1}}\right),
     \end{aligned}
$$
where $$ h_{j}\left(\frac{l}{p^{s+1}}\right) = h_{j}\left(\frac{k+rp^{s}}{p^{s+1}}\right) = \sum_{\nu=1}^{p-1} \alpha_{\nu,k}^{j} e^{2\pi i\nu r/p}, $$ for $ r = 0,1,\ldots, p-1,\, k = 0,1, \ldots, p^{s}-1 $.
\end{remark}
The following theorem gives a characterization for $ p $-vanishing moments of Haar-type wavelets.
\begin{theorem}\label{th4.4.2}
 Let $ \psi_{j}, \, j = 1,\ldots, p-1 $\, be the wavelet finctions described in theorem \ref{th4.4.1}. then $ \psi_{j} $\, always have $ 1 $\, $ p $-vanishing moment. Moreover, $ \psi_{j} $\, has infinite number of $ p $-vanishing moments if and only if $ h_{j}(0) = 0, \, j = 1,\ldots, p-1 $. 
\end{theorem}
\begin{proof}
 We have $$ \psi_{j}(x) =  \sum_{l=0}^{p^{s+1}-1} h_{j}\left (\frac{l}{p^{s+1}}\right) \phi\left(\frac{x}{p} - \frac{l}{p^{s+1}}\right), $$
 where $$ h_{j}\left(\frac{l}{p^{s+1}}\right) = h_{j}\left(\frac{k+rp^{s}}{p^{s+1}}\right) = \sum_{\nu=1}^{p-1} \alpha_{\nu,k}^{j} e^{2\pi i\nu r/p}, $$ for $ r = 0,1,\ldots, p-1,\, k = 0,1, \ldots, p^{s}-1 $.
 Then 
 $$ \begin{aligned}
     \sum_{l=0}^{p^{s+1}-1} h_{j}\left(\frac{l}{p^{s+1}}\right) & = \sum_{k=0}^{p^{s}-1} \sum_{r=0}^{p-1} \sum_{\nu=1}^{p-1} \alpha_{\nu,k}^{j} e^{2\pi i\nu r/p} \\
     & = \sum_{k=0}^{p^{s}-1} \sum_{\nu=1}^{p-1} \alpha_{\nu,k}^{j} \left(\sum_{r=0}^{p-1}e^{2\pi i\nu r/p}\right) \\
     & = \sum_{k=0}^{p^{s}-1} \sum_{\nu=1}^{p-1} \alpha_{\nu,k}^{j} \times 0 \\
     & = 0,
    \end{aligned}
$$
for $ j = 1,\ldots, p-1 $. That is $ \psi_{j} $\, always have $ 1 $\, $ p $-vanishing moment. \par 
We can write any $ \frac{k}{p^{s+1}} $, for $ k = 0,1,2, \ldots, p^{s}-1 $\, in the form 
$$ \frac{k}{p^{s+1}} = \frac{k_{0}}{p^{s+1}} + \frac{k_{1}}{p^{s}} + \cdots + \frac{k_{s-1}}{p^{2}}, \, k_{0},k_{1},\ldots, k_{s-1} \in \{ 0,1,\ldots, p-1\}. $$
Let $ M_{\gamma} $\, for $ \gamma = 2,\ldots, s $\, be the set given by, 
\begin{equation*}
 \begin{split}
  M_{\gamma} = \{ k = k_{0}+k_{1}p & + k_{2}p^{2}+ \cdots k_{s-1}p^{s-1} : k_{0} = k_{1} = \cdots = k_{s-\gamma} = 0,\\ & k_{s-\gamma+1}, k_{s-\gamma+2},\ldots, k_{s-1} = 0,1,\ldots, p-1 ,\, k_{s-\gamma+1} \neq 0\},
 \end{split}
\end{equation*}  
and $$ M_{s+1} = \{ k = k_{0}+k_{1}p + k_{2}p^{2}+ \cdots k_{s-1}p^{s-1} : k_{0}, k_{1}, \ldots, k_{s-1} = 0,1,\ldots, p-1 ,\, k_{0} \neq 0\}. $$ Then 
$$ \begin{aligned}
     \sum_{l=1}^{p^{s+1}-1} \left \lvert \frac{l}{p^{s+1}} \right \rvert_{p}^{\mu} h_{j}\left (\frac{l}{p^{s+1}}\right) = & \sum_{k=0}^{p^{s}-1} \sum_{r=0}^{p-1}\left \lvert \frac{k+rp^{s}}{p^{s+1}} \right \rvert_{p}^{\mu}h_{j}\left(\frac{k+rp^{s}}{p^{s+1}}\right)\\ = & \sum_{r=1}^{p-1}p^{\mu}\sum_{\nu=1}^{p-1}\alpha_{\nu,0}^{j} e^{2\pi i\nu r/p} \\ & +  \sum_{k=1}^{p^{s}-1} \sum_{r=0}^{p-1} \sum_{\nu=1}^{p-1} \left \lvert \frac{k+rp^{s}}{p^{s+1}} \right \rvert_{p}^{\mu}\alpha_{\nu,k}^{j} e^{2\pi i\nu r/p}\\
      = & \, p^{\mu} \sum_{\nu=1}^{p-1}\alpha_{\nu,0}^{j}\left( \sum_{r=1}^{p-1}e^{2\pi i\nu r/p} \right) \\
      & + p^{2\mu} \sum_{k \in M_{2}} \sum_{\nu = 1}^{p-1}\alpha_{\nu,k}^{j} \left( \sum_{r=0}^{p-1}e^{2\pi i\nu r/p} \right) \\ & + p^{3\mu} \sum_{k \in M_{3}} \sum_{\nu = 1}^{p-1}\alpha_{\nu,k}^{j} \left( \sum_{r=0}^{p-1}e^{2\pi i\nu r/p} \right) \\ & + \cdots + p^{(s+1)\mu} \sum_{k \in M_{s+1}} \sum_{\nu = 1}^{p-1}\alpha_{\nu,k}^{j} \left( \sum_{r=0}^{p-1}e^{2\pi i\nu r/p} \right) \\ = & \, -p^{\mu} \sum_{\nu=1}^{p-1}\alpha_{\nu,0}^{j} = -h_{j}(0).
   \end{aligned}
$$
Thus, $ \psi_{j} $\, has infinite number of $ p $-vanishing moments if and only if $ h_{j}(0) = 0, \, j = 1,\ldots, p-1 $. 
\end{proof}

\section{\texorpdfstring{$ p $}p-Vanishing Moments of non-Haar type Wavelets}

In \cite{ks}, wavelet bases different from those described above were constructed; these bases were called non-Haar bases.
\begin{theorem}\label{th4.5.1}\cite{ks}
 Let \begin{equation}\label{eq4.5.1}
      \begin{split}
      J_{p,m} = \{ s = p^{-m}(s_{0}+s_{1}p+s_{2}p^{2}+ \cdots & + s_{m-1}p^{m-1}) : s_{j} = 0,1,\ldots, p-1;\\ & j = 0,1,\ldots, m-1;\, s_{0} \neq 0 \},
      \end{split}
     \end{equation}
where $ m \geq 1 $\, is a fixed positive integer. Consider the set of $ ( p - 1 ) p^{m - 1} $\, functions 
\begin{equation}\label{eq4.5.2}
 \theta_{s}^{(m)} (x) = \chi(x,s)1_{B_{0}(0)}(x), \quad s \in J_{p,m},\, x \in \mathbb{Q}_{p},
\end{equation}
and the family of functions generated by their dilations and translations:
\begin{equation}\label{eq4.5.3}
 \theta_{s;ja}^{(m)} (x) = p^{-j/2} \theta_{s}^{(m)}(p^{j}x-a), \quad s \in J_{p,m},\, x \in \mathbb{Q}_{p},\, j \in \mathbb{Z},\, a \in I_{p},
\end{equation}
where $ 1_{B_{0}(0)} $\, is the indicator function of $ B_{0}(0) $. Then the functions \eqref{eq4.5.3} form an orthonormal 
$ p $-adic wavelet basis in $ L_{2}(\mathbb{Q}_{p}) $\, which is non-Haar type for $ m\geq 2 $.
\end{theorem}
\begin{theorem}\label{th4.5.2}\cite{ks}
 For any fixed $ \nu = 1 , 2 , \ldots $, the functions
 \begin{equation}\label{eq4.5.4}
  \psi_{s}^{(m),\nu} (x) = \sum_{k=0}^{p^{\nu}-1} \alpha_{s,k} \theta_{s}^{(m)}\left(x-\frac{k}{p^{\nu}}\right), \quad s \in J_{p,m},
 \end{equation}
are wavelet functions if and only if
\begin{equation}\label{eq4.5.5}
 \alpha_{s,k} = p^{-\nu} \sum_{r=0}^{p^{\nu}-1} \gamma_{s,r} e^{-2 \pi i \frac{-s+r}{p^{\nu}} k},
\end{equation}
where $ \gamma_{s,k} \in \mathbb{C},\, \lvert \gamma_{s,k} \rvert = 1, \, k = 0,1,\ldots, p^{\nu}-1, \, s \in J_{p,m} $.
\end{theorem}
 The following two theorems gives the connection between $ p $-vanishing moment of non-Haar type wavelet functions and the approximation order of the indicator function of the compact open subgroup $ B_{0}(0) $\, of $ \mathbb{Q}_{p} $.
\begin{theorem}\label{th4.5.3}
 Let $ \phi = 1_{B_{0}(0)} $\, and $ \theta_{s}^{(m)} $\, be defined by \eqref{eq4.5.3}. If $ \phi $\, has approximation order $ k $, then $ \theta_{s}^{(m)} $\, has $ k $\, $ p $-vanishing moments. 
\end{theorem}
 \begin{proof}
  We have for a fixed integer $ m > 2 $, 
$$ \theta_{s}^{(m)} (x) = \chi(x,s)1_{B_{0}(0)}(x), \quad s \in J_{p,m},\, x \in \mathbb{Q}_{p}. $$
That is, $ \theta_{s}^{(m)} (x) = \chi(x,s)\phi(x) $. Now, 
$$ \begin{aligned}
    \widehat{\theta_{s}^{(m)}}(\xi) & = \int_{\mathbb{Q}_{p}} \chi(x,s)\phi(x) \chi(x, \xi) dx \\
    & = \int_{\mathbb{Q}_{p}} \phi(x)\chi(x, \xi+s) dx \\
    & = \widehat{\phi}(\xi +s).
   \end{aligned}
$$
Hence $ D^{\mu}\widehat{\theta_{s}^{(m)}}(0) = D^{\mu}\widehat{\phi}(s) $. \par 
From theorem \ref{th1.4}, $ \phi $\, has $ k $\, approximation order if and only if $ D^{\mu}\widehat{\phi}(\alpha) = 0 $\, for all $ \alpha \in I_{p} \setminus \{ 0 \}, \, \lvert \mu \rvert \leq k-1 $. Here $ I_{p} \setminus \{ 0 \} = \bigcup_{m \geq 1} J_{p,m} $. \par 
We have $ \{ \theta_{s}^{(m)} \}_{s \in J_{p,m}} $\, has $ k $\, $ p $-vanishing moments if and only if $ D^{\mu}\widehat{\theta_{s}^{(m)}}(0) = 0 $\, for all $ s \in J_{p,m}, \, 0 \leq \mu \leq k-1 $. That is, if and only if $ D^{\mu}\widehat{\phi}(s) = 0 $\, for all $ s \in J_{p,m}, \, 0 \leq \mu \leq k-1 $. \par 
Since $ m $\, is fixed, we can conclude that if $ \phi $\, has approximation order $ k $, then $ \theta_{s}^{(m)} $\, has $ k $\, $ p $-vanishing moments. 
 \end{proof}

 \begin{theorem}\label{th4.5.4}
   Let $ \phi = 1_{B_{0}(0)} $\, and $ \psi_{s}^{(m),\nu} $\, be the wavelet functions defined by \eqref{eq4.5.4}. If $ \phi $\, has approximation order $ k $, then $ \psi_{s}^{(m),\nu} $\, has $ k $\, $ p $-vanishing moments. 
 \end{theorem}
 \begin{proof}
  For any fixed $ \nu = 1 , 2 , \ldots $, let $ \psi_{s}^{(m),\nu} $\, be the wavelet functions defined by \eqref{eq4.5.4}, $$ \psi_{s}^{(m),\nu} (x) = \sum_{k=0}^{p^{\nu}-1} \alpha_{s,k} \theta_{s}^{(m)}\left(x-\frac{k}{p^{\nu}}\right), \quad s \in J_{p,m}, $$
with 
 $$ \alpha_{s,k} = p^{-\nu} \sum_{r=0}^{p^{\nu}-1} \gamma_{s,r} e^{-2 \pi i \frac{-s+r}{p^{\nu}} k}, $$
where $ \gamma_{s,k} \in \mathbb{C},\, \lvert \gamma_{s,k} \rvert = 1, \, k = 0,1,\ldots, p^{\nu}-1, \, s \in J_{p,m} $.
That is, $$ \psi_{s}^{(m),\nu} (x) = \sum_{k=0}^{p^{\nu}-1} \alpha_{s,k}  \chi\left(x-\frac{k}{p^{\nu}}, s \right)\phi\left(x-\frac{k}{p^{\nu}}\right), \quad s \in J_{p,m}. $$ Now,
$$ \begin{aligned}
    \widehat{\psi_{s}^{(m),\nu}}(\xi) & = \int_{\mathbb{Q}_{p}} \sum_{k=0}^{p^{\nu}-1} \alpha_{s,k} \chi\left(x-\frac{k}{p^{\nu}}, s \right)\phi\left(x-\frac{k}{p^{\nu}}\right) \chi(x, \xi) dx \\
    & =\sum_{k=0}^{p^{\nu}-1} \alpha_{s,k} \int_{\mathbb{Q}_{p}} \phi(t) \chi(t,s)\chi\left(t+\frac{k}{p^{\nu}}, \xi \right) dt \\
    & =\sum_{k=0}^{p^{\nu}-1} \alpha_{s,k} \chi\left(\frac{k}{p^{\nu}}, \xi \right)\int_{\mathbb{Q}_{p}} \phi(t) \chi(t,s) \chi(t, \xi) dt \\
    & = \sum_{k=0}^{p^{\nu}-1} \alpha_{s,k} \chi\left(\frac{k}{p^{\nu}}, \xi \right)\int_{\mathbb{Q}_{p}} \phi(t) \overline{\chi(t, \xi+s)} dt \\
    & = \sum_{k=0}^{p^{\nu}-1} \alpha_{s,k} \chi\left(\frac{k}{p^{\nu}}, \xi \right)\widehat{\phi}(\xi +s).
   \end{aligned}
$$
Let $ f(\xi)= \chi\left(\frac{k}{p^{\nu}}, \xi \right) $\, and $ g(\xi)= \widehat{\phi}(\xi +s) $. Then applying the Leibniz formula for diﬀerentiation,
$$ \begin{aligned}
    D^{\mu}\widehat{\psi_{s}^{(m),\nu}}(\xi) & = \sum_{k=0}^{p^{\nu}-1} \alpha_{s,k} \sum_{\beta \leq \mu} {\mu \choose \beta} D^{\beta}f(\xi) D^{\mu-\beta}g(\xi) \\
    & = \sum_{k=0}^{p^{\nu}-1} \alpha_{s,k} \sum_{\beta \leq \mu} {\mu \choose \beta} p^{\nu\beta} \chi\left(\frac{k}{p^{\nu}}, \xi \right) D^{\mu-\beta}\widehat{\phi}(\xi+s). 
   \end{aligned}
$$
  That is, $$ D^{\mu}\widehat{\psi_{s}^{(m),\nu}}(0) = \sum_{k=0}^{p^{\nu}-1} \alpha_{s,k} \sum_{\beta \leq \mu} {\mu \choose \beta} p^{\nu\beta} D^{\mu-\beta}\widehat{\phi}(s). $$ 
  From theorem \ref{th1.4}, $ \phi $\, has $ k $\, approximation order if and only if $ D^{\mu}\widehat{\phi}(\alpha) = 0 $\, for all $ \alpha \in I_{p} \setminus \{ 0 \}, \, \lvert \mu \rvert \leq k-1 $. Here $ I_{p} \setminus \{ 0 \} = \bigcup_{m \geq 1} J_{p,m} $. \par 
We have $ \{ \psi_{s}^{(m),\nu} \}_{s \in J_{p,m}} $\, has $ k $\, $ p $-vanishing moments if and only if $ D^{\mu}\widehat{\psi_{s}^{(m),\nu}}(0) = 0 $\, for all $ s \in J_{p,m}, \, 0 \leq \mu \leq k-1 $. That is, if and only if $ D^{\mu-\beta}\widehat{\phi}(s) = 0 $\, for all $ s \in J_{p,m}, \, 0 \leq \mu \leq k-1,\,\beta \leq \mu $. \par 
Since $ m $\, is fixed, we can conclude that if $ \phi $\, has approximation order $ k $, then $ \psi_{s}^{(m),\nu} $\, has $ k $\, $ p $-vanishing moments. 
 \end{proof}
\begin{remark}
 Let $ \phi = 1_{B_{0}(0)} $\, and $ \psi $\, be a non-Haar type wavelet function. If $ \phi $\, has approximation order $ k $, then $ \psi $\, has $ k $\, $ p $-vanishing moments. 
\end{remark}

\section{\texorpdfstring{$ p $}p-vanishing moments of nonorthogonal wavelets}

In \cite{aes3}, the authors developed a method for constructing MRA-based p-adic wavelet systems that form Riesz bases in $ L_{2}(\mathbb{Q}_{p}) $. The explicit construction is as follows:\par 
 For an integer $ K \geq 0 $, we set 
 $$ A_{K} = \{ \frac{ap^{j}-b}{p^{j+1}}, \frac{p^{K+1}-b}{p^{K+1}}:\, j=1,\ldots,K,\, a,b = 1,\ldots, p-1 \},  $$
 $$ B_{K} = \{ 0, \frac{p^{j}-b}{p^{j}}:\, j=1,\ldots,K,\, b = 1,\ldots, p-1 \} $$
 It is easily seen that $ \# A_{K} = (p-1)\# B_{K} $\, and $ \# B_{K} = 1+K(p-1) $\, for all $ K $. Moreover, $ A_{K}, B_{K} \subset I_{p} $\, and $$ A_{K}\cap B_{K} = \emptyset, \quad A_{K} \cup B_{K} = \bigcup_{j=0}^{p-1}\frac{1}{p}(j+B_{K}). $$ 
 Set $ \chi_{p}(\xi)=e^{2\pi i \{\xi\}_{p}} $. Let us deﬁne trigonometric polynomials $ m_{0} = m_{0,K} $\, and $ n_{0} = n_{0,K} $\, of degrees $ (p - 1)(1 + (p - 1)K) $\, and $ 1 + (p - 1)K $, respectively, by
\begin{equation}\label{eq4.6.1}
m_{0,K}(\xi) = \frac{1}{p} \prod_{r \in A_{K}} \left( \chi_{p}(\xi)-\chi_{p}(r) \right), \qquad n_{0,K}(\xi) = \frac{1}{p} \prod_{r \in B_{K}}\left( \chi_{p}(\xi)-\chi_{p}(r) \right).
\end{equation}
One can easily verify that $ m_{0,K}(0) = 1 $\, for all $ K \geq 0 $. \par 
Given integers $ M \geq 0 $\, and $ N \geq 0 $, we deﬁne the Fourier transform of functions $ \phi = \phi_{M,N} $\, and $ \psi = \psi_{M,N} $\, by
\begin{equation}\label{eq4.6.2}
\widehat{\phi}_{M,N}(\xi) = \prod_{j=0}^{\infty} m_{0,K}(\frac{\xi}{p^{N-j}}), \qquad \xi \in \mathbb{Q}_{p},
\end{equation}
\begin{equation}\label{eq4.6.3}
\widehat{\psi}_{M,N}(\xi) = n_{0,K}(\frac{\xi}{p^{N}}) \widehat{\phi}_{M,N}(p\xi), \qquad \xi \in \mathbb{Q}_{p}.
\end{equation}
where $ K=M+N $.
\begin{lemma}\label{lm4.6.1}\cite{aes3}
For all $ M $\, and $ N $, the following statements hold:
\begin{enumerate}
\item $ \widehat{\phi}_{M,N} \in \mathcal{D}_{M}^{N} $\, and $ \widehat{\psi}_{M,N} \in \mathcal{D}_{M+1}^{N} $,
\item $ \widehat{\phi}_{M,N}(\frac{l}{p^{M}}) \neq 0 $\, if and only if $ l \in p^{K}B_{K},\, 0 \leq l \leq p^{K}-1 $, 
\item $ \widehat{\psi}_{M,N}(\frac{l}{p^{M+1}}) \neq 0 $\, if and only if $ l \in p^{K+1}A_{K},\, 0 \leq l \leq p^{K+1}-1 $.
\end{enumerate}
\end{lemma}
Set \begin{equation}\label{eq4.6.4}
 \psi^{(\nu)}(x) = \psi(x - \nu + 1),\, \nu = 1,\ldots , p - 1. 
 \end{equation}
\begin{theorem}\label{th4.6.1}\cite{aes3}
For integers $ M, N \geq 0 $, the function $ \phi_{M,N} $\, deﬁned by its Fourier transform \eqref{eq4.6.2} generates an MRA if and only if $ M \leq \frac{p^{N}-1}{p-1}-N $. Moreover, in this case, $ \psi^{(\nu)}_{M,N},\, \nu = 1,\ldots ,p-1 $\, is a set of wavelet functions, and the corresponding wavelet system $ \{ \psi^{(\nu)}_{ja}(x) := p^{j/2} \psi^{(\nu)}(p^{-j}x - a) : \nu = 1,\ldots , r,\, a \in I_{p},\, j \in \mathbb{Z} $\, forms a Riesz basis for $ L_{2}(\mathbb{Q}_{p}) $\, if and only if $ M = \frac{p^{N}-1}{p-1}-N $. 
\end{theorem}
It is easy to see that if $ (M, N ) = (0, 0) $\, or $ (M, N ) = (0, 1) $, then $ \phi_{M,N}(x) = \Omega(\lvert x \rvert_{p}) $\, and we obtain the Haar MRA. On the contrary, for $ N > 1 $\, and $ M = \frac{p^{N}-1}{p-1} - N $, the functions $ \phi_{M,N} $\, generate pairwise distinct MRAs and each of these scaling functions is not orthogonal, which leads to nonorthogonal wavelet Riesz bases. 
\begin{example}
Let $ p=2, \, N=2 $\, and $ M=1 $, then $ K =3 $ and $$ A_{3} = \{\frac{1}{4},\frac{3}{8},\frac{7}{16}, \frac{15}{16} \} \text{   and   } B_{3} = \{0,\frac{1}{2},\frac{3}{4}, \frac{7}{8} \}. $$
Then $$ m_{0,K}(\xi) = \frac{1}{p}\left( \chi_{p}(\xi)-\chi_{p}(\frac{1}{4}) \right) \left( \chi_{p}(\xi)-\chi_{p}(\frac{3}{8}) \right) \left( \chi_{p}(\xi)-\chi_{p}(\frac{7}{16}) \right)\left( \chi_{p}(\xi)-\chi_{p}(\frac{15}{16}) \right), $$ and 
$$ n_{0,K}(\xi) = \frac{1}{p}\left( \chi_{p}(\xi)-1 \right) \left( \chi_{p}(\xi)-\chi_{p}(\frac{1}{2}) \right) \left( \chi_{p}(\xi)-\chi_{p}(\frac{3}{4}) \right)\left( \chi_{p}(\xi)-\chi_{p}(\frac{7}{8}) \right). $$
Then the refinable function is given by \eqref{eq4.6.2} and the corresponding wavelets function is $ \psi^{(\nu)}_{M,N} $\, where $ \psi_{M,N} $\, is given by \eqref{eq4.6.3}.
\end{example}
The following theorem characterizes the $ p $-vanishing moment of nonorthogonal MRA wavelets.
\begin{theorem}\label{th4.6.2}
 Let $ \psi^{(\nu)}_{M,N} $\, be the wavelet functions defined by \eqref{eq4.6.4}. Then $ \psi^{(\nu)}_{M,N} $\, has $ k $\, $ p $-vanishing moments if and only if $ D^{\mu}n_{0,K}(0) = 0, \, \forall 0 \leq \mu \leq k-1 $.
\end{theorem}
\begin{proof}
We have $ \psi^{(\nu)}(x) = \psi(x - \nu + 1) $. Then the Fourier transform of $ \psi^{(\nu)}_{M,N} $\, is given by
$$ \begin{aligned}
 \widehat{\psi^{(\nu)}}_{M,n}(\xi) & = \int_{\mathbb{Q}_{p}} \psi^{(\nu)}_{M,N}(x) \chi(x, \xi) dx \\
    & = \int_{\mathbb{Q}_{p}} \psi_{M,N}(x-\nu+1)\chi(x, \xi) dx \\
    & = \int_{\mathbb{Q}_{p}} \psi_{M,N}(t)\chi(t+\nu-1, \xi) dx \\
    & = \chi(\nu-1,\xi) \int_{\mathbb{Q}_{p}} \psi_{M,N}(t)\chi(t, \xi) dx \\
    & = \chi(\nu-1,\xi) \widehat{\psi}_{M,N}(\xi).
   \end{aligned} $$
 Then applying the Leibniz formula for diﬀerentiation,
$$ D^{\mu}\widehat{\psi}^{(\nu)}_{M,N}(\xi) = \sum_{\beta \leq \mu} {\mu \choose \beta} D^{\beta}\chi(\nu-1,\xi) D^{\mu-\beta}\widehat{\psi}_{M,N}(\xi). $$
Also $ \widehat{\psi}_{M,N}(\xi) = n_{0,K}(\frac{\xi}{p^{N}}) \widehat{\phi}_{M,N}(p\xi) $. By the Leibniz formula for diﬀerentiation,
$$ D^{\mu}\widehat{\psi}_{M,N}(\xi) = \sum_{\beta \leq \mu} {\mu \choose \beta} D^{\beta}n_{0,K}(\frac{\xi}{p^{N}})D^{\mu-\beta}\widehat{\phi}_{M,N}(\xi).  $$ By Lemma \ref{lm4.6.1} we have, $ \phi_{M,N}(0) \neq 0 $.
Thus we can see that $ \psi^{(\nu)}_{M,N} $\, has $ k $\, $ p $-vanishing moments if and only if $ D^{\mu}n_{0,K}(0) = 0, \, \forall 0 \leq \mu \leq k-1 $.
\end{proof}

\section{Conclusion}
The definitions of $ p $-vanishing moments of compactly supported functions on $ \mathbb{Q}_{p} $\, and discrete $ p $-vanishing moments of finitely supported sequences on $ I_{p} $\, are given in this work. The relationship between the $ p $-vanishing moments and discrete $ p $-vanishing moments are established. The $ p $-vanishing moments of Haar-type and non-Haar type wavelet functions are calculated. We proved the connection between $ p $-vanishing moment of non-Haar type wavelet functions and the approximation order of the indicator function of the compact open subgroup $ B_{0}(0) $\, of $ \mathbb{Q}_{p} $. Finally, we characterized the $ p $-vanishing moments of nonorthogonal wavelets.

\section*{Acknowledgement}
We are very grateful to the authors of the articles in the references.

\bibliographystyle{IEEEtran}
\bibliography{article}

\end{document}